\documentclass{amsart}
\usepackage{amssymb, graphicx}
\usepackage[mathscr]{eucal}
\usepackage{slashbox}
\usepackage{rotating}
\usepackage{diagbox}

\theoremstyle{plain}
\newtheorem{prop}{Proposition}[section]
\newtheorem{coro}[prop]{Corollary}

\newtheorem{conj}[prop]{Conjecture}
\newtheorem{lemm}[prop]{Lemma}
\newtheorem{ques}[prop]{Question}
\newtheorem{thm}[prop]{Theorem}

\theoremstyle{definition}
\newtheorem{defn}[prop]{Definition}

\newtheorem{rem}[prop]{Remark}

\DeclareMathOperator{\breadth}{breadth}

\DeclareMathOperator{\Ker}{Ker}
\DeclareMathOperator{\Image}{Image}
\DeclareMathOperator{\maxdeg}{maxdeg}
\DeclareMathOperator{\mindeg}{mindeg}
\DeclareMathOperator{\lk}{lk}
\DeclareMathOperator{\Kh}{Kh}
\DeclareMathOperator{\Le}{Le}

\def\mcg#1;#2{\Gamma_{#1,#2}}
\def\fg#1;#2{\Pi_{#1,#2}}
\def\tb#1;#2{\mathscr{K}_{\frac{#1}{#2}}}

\begin{document}

\title[On Khovanov Homology of Quasi-Alternating Links]
{On Khovanov Homology of Quasi-Alternating Links}


\author{Khaled Qazaqzeh}
\address{Address:- Department of Mathematics, Faculty of Science,  Kuwait
University
P. O. Box 5969 Safat-13060, Kuwait, State of Kuwait}

\email{khaled.qazaqzeh@ku.edu.kw}

\author{Nafaa Chbili}
\address{Department of Mathematical Sciences\\ College of Science\\ UAE University
\\ 15551 Al Ain, U.A.E.}
\email{nafaachbili@uaeu.ac.ae}
\urladdr{http://faculty.uaeu.ac.ae/nafaachbili}

\date{15/03/2021}

\begin{abstract}
We prove that the length of any gap in the differential grading of the Khovanov homology of any
quasi-alternating link is one. As a consequence, we obtain that the length of any gap in the
Jones polynomial of any such link is one. This establishes a weaker version of
Conjecture 2.3 in \cite{CQ}. Moreover, we obtain a lower bound for the determinant of any such link
in terms of the breadth of its Jones polynomial. This establishes a weaker
version of Conjecture 3.8 in \cite{QC}. The main tool in obtaining this result is
establishing the Knight Move Conjecture \cite[Conjecture\,1]{Ba} for the class of quasi-alternating links.

\end{abstract}

\maketitle

\section{introduction}

The class of alternating links has played an important role in the development
of knot theory since its early age. In particular, the study of their Jones polynomials led to the solution of long-lasting conjectures in knot theory.
Thistletwaite \cite{Th} proved that the Jones polynomial of any prime alternating link, which is not a  $(2,p)$-torus link,
is alternating and has no gaps. Moreover, the  breadth  of the Jones polynomial of any connected alternating link is equal to  its crossing number. Alternating links are also known to have simple Khovanov and link Floer  homologies.
Indeed, the Khovanov homology of a given alternating link $L$ is entirely  determined
by its  signature $\sigma_L$ and its  Jones polynomial $V_L(t)$.  Similarly, the  link Floer  homology  of any alternating link $L$ is entirely  determined
by its  signature $\sigma_L$ and its  Alexander  polynomial $\Delta_L(t)$. Furthermore, Ozsv$\acute{a}$th and
Szab$\acute{o}$ \cite{OS} studied the Heegaard Floer Homology  of the  branched double-cover of alternating links and proved that this homology is  determined by the determinant of the link, $\rm{det}(L)$.
This  homological property    extends to  a larger family of links called
quasi-alternating.  Unlike alternating links which admit   a simple diagrammatic definition, quasi-alternating links are  defined recursively as follows:
\begin{defn}\label{def}
The set $\mathcal{Q}$ of quasi-alternating links is the smallest set
satisfying the following properties:
\begin{itemize}
	\item The unknot belongs to $\mathcal{Q}$.
  \item If $L$ is a link with a diagram $D$ containing a crossing $c$ such that
\begin{enumerate}
\item both smoothings of the diagram $D$ at the crossing $c$, $L_{0}$ and $L_{1}$ as
in Figure \ref{figure} belong to $\mathcal{Q}$, and
\item $\det(L_{0}), \det(L_{1}) \geq 1$,
\item $\det(L) = \det(L_{0}) + \det(L_{1})$; then $L$ is in $\mathcal{Q}$ and in this
case we say $L$ is quasi-alternating at the crossing $c$ with quasi-alternating
diagram $D$.
\end{enumerate}
\end{itemize}
\end{defn}

\begin{figure} [h]
\begin{center}
\includegraphics[scale=0.4]{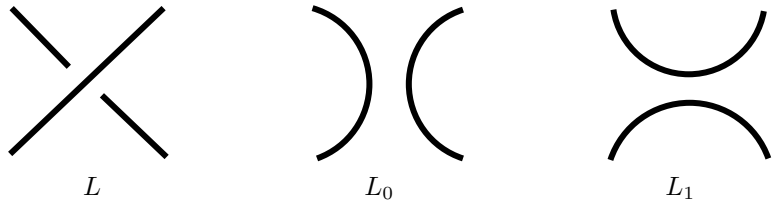} \\
{$L$}\hspace{3.5cm}{$L_0$}\hspace{3.6cm}$L_{1}$
\end{center}
\caption{The link diagram $L$ at the crossing $c$ and its smoothings $L_{0}$ and
$L_{1}$ respectively.}\label{figure}
\end{figure}
It is impossible to use the above definition to show that a given link is not quasi-alternating. As an alternative, several obstructions for a link to be quasi-alternating have been introduced through
the past two decades. Many of these  have been established  for alternating links first, then extended  to this new class of links. Some of these  main  obstructions  are listed here.

\begin{enumerate}
\item the branched double-cover of any quasi-alternating link
is an $L$-space \cite[Proposition.\,3.3]{OS};
\item the space of branched double-cover of any
quasi-alternating link bounds a negative definite $4$-manifold $W$ with
$H_{1}(W) = 0$ \cite[Proof of Lemma.\,3.6]{OS};
\item the $\mathbb{Z}/2\mathbb{Z}$ knot Floer homology group of any
quasi-alternating link is $\sigma$-thin \cite[Theorem.\,2]{MO};
\item the reduced ordinary Khovanov homology group of any
quasi-alternating link is $\sigma$-thin \cite[Theorem.\,1]{MO};
 \item the reduced odd Khovanov homology group of any
quasi-alternating link is $\sigma$-thin \cite[Remark after Proposition.\,5.2]{ORS};
 \item the determinant of any quasi-alternating link is bigger  than  the degree of
its $Q$-polynomial \cite[Theorem\,2.2]{QC}. This inequality was  sharpened later to the determinant
minus one is bigger  than or equal to the degree of the $Q$-polynomial with
equality holds only for $(2,n)$-torus links \cite[Theorem\,1.1]{T}.
\end{enumerate}

The main purpose of this paper is to introduce a new and simple obstruction for a link to be quasi-alternating in terms of Khovanov homology.   More precisely, we prove  that if the differential grading of the
Khovanov homology of any quasi-alternating link  has a gap, then the length of this gap
is one. As an immediate application, we obtain that any gap in the Jones polynomial of any
quasi-alternating link has length one. Moreover, a lower bound for the determinant of any such
link is obtained in terms of  the breadth of its  Jones polynomial. The main result of this paper totally relies on establishing the Knight Move Conjecture  \cite[Conjecture\,1]{Ba} for the class of quasi-alternating links.

Consequently, we prove that certain links are not quasi-alternating as the differential grading of their Khovanov homology has a gap of length bigger than one. On the other hand, we show that some subclasses of quasi-alternating links have no
gaps in the differential grading of their Khovanov homology. This leads us to suggest
Conjecture \ref{conjmain} that implies both Conjecture 2.3 in \cite{CQ} and Conjecture
3.8 in \cite{QC}.

This paper is organized as follows. In Section 2, we introduce  some background and notations needed for the rest of the paper. In Section 3, we give the proof of our main result. Finally, Section 4 will be devoted to give some applications and consequences of our main result.

\section{Background and Notations}
In this section, we introduce some notations  and definitions,  and we  review some properties of Khovanov homology  needed in the this paper. Without loss of generality,
we assume that the three unoriented links $L, L_{0}$ and $L_{1}$ are given according to
the scheme in Figure \ref{figure}. The crossing $c$ of the link $L$ will be called of type I if it looks like the crossing in Figure \ref{figure}, otherwise it will be of type II. For an oriented link $L$, we always assume that we smooth a positive crossing according to the scheme in Figure \ref{Diagram1}.

All the results of this paper discuss the case where the crossing $c$ is of type I and oriented positively unless mentioned otherwise.  Similar results can be obtained in the other cases by taking the mirror image of $L$ if it is required.

\subsection{The Jones polynomial}
The Jones polynomial $V_L(t)$ is an invariant of oriented links. It is a Laurent
polynomial with integral coefficients that can be defined in several ways. In this
subsection, we shall briefly recall the definition of  this polynomial in terms of the Kauffman
bracket  and review some of its properties needed in the sequel.

\begin{defn}
The Kauffman bracket polynomial is a function from the set of unoriented link
diagrams in the oriented plane to the ring of Laurent polynomials with integer
coefficients in an indeterminate $A$. It maps a link $L$ to $\left\langle L
\right\rangle\in \mathbb Z[A^{-1},A]$ and it is defined by the following relations:
\begin{enumerate}
\item $\left\langle \bigcirc \right\rangle=1$,
\item $\left\langle \bigcirc \cup L\right\rangle=(-A^{-2}-A^2)\left\langle L
\right\rangle$,
\item $\left\langle L\right\rangle=A\left\langle L_0\right\rangle+A^{-1}\left\langle
L_1\right\rangle$,
\end{enumerate}
where $\bigcirc$, in relation 2 above,  denotes a trivial circle  disjoint from the rest of the link, and  $L,L_0, \text{and } L_1$  represent three unoriented links
which are identical everywhere  except in a small region where they are as indicated in  Figure \ref{figure}.
\end{defn}
Given an oriented link diagram  $L$, we let $x(L)$ to denote the number of negative crossings and  $y(L)$ to denote the number of positive crossings in $L$ according to the scheme in Figure \ref{Diagram1}.
The writhe of link diagram $L$ is defined to be the integer $w(L) = y(L) - x(L)$.
\begin{defn}
The Jones polynomial $V_{L}(t)$ of an oriented link $L$ is the Laurent polynomial
in $t^{\pm 1/2}$ with integer coefficients defined by
\begin{equation*}
V_{L}(t)=((-A)^{-3w(L)}\left\langle L \right\rangle)_{t^{1/2} = A^{-2}}\in
\mathbb Z[t^{-1/2},t^{1/2}],
\end{equation*}
where $\left\langle L \right\rangle $ is the Kauffman bracket of the unoriented link obtained from $L$ by ignoring  the orientation.
\end{defn}

\begin{figure}[h]
	\centering
		\reflectbox{\includegraphics[scale=0.11]{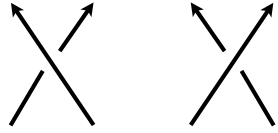}}\hspace{1.1cm}
		\includegraphics[scale=0.11]{Diagram1}
	\caption{Negative and positive crossings respectively.}
	\label{Diagram1}
\end{figure}

\subsection{Khovanov Invariant}
This link invariant is a bigraded cohomology theory, with rational coefficients, that was first introduced
by Khovanov in \cite{Kh1}.
For an oriented link $L$, we denote the  $i$-{th} homology group of the complex
$\overline{\mathcal{C}}(L)$ and $\mathcal{C}(L)$ by $\overline{\mathcal{H}}^{i}(L)$
and $\mathcal{H}^{i}(L)$, respectively. Also, we denote the $j$-th graded component
of $\overline{\mathcal{H}}^{i}(L)$, and $\mathcal{H}^{i}(L)$ by
$\overline{\mathcal{H}}^{i,j}(L)$, and $\mathcal{H}^{i,j}(L)$, respectively.
Therefore, we have
\[
\overline{\mathcal{H}}^{i}(L) = \oplus_{j\in \mathbb{Z}}
\overline{\mathcal{H}}^{i,j}(L), \ \text{and} \
\mathcal{H}^{i}(L) = \oplus_{j\in \mathbb{Z}}\mathcal{H}^{i,j}(L).
\]

The homology groups $\overline{\mathcal{H}}(L)$ and $\mathcal{H}(L)$ are isomorphic  up to some
shifts. More precisely

\begin{equation}\label{relation}
\mathcal{H}^{i, j}(L) = \overline{\mathcal{H}}^{i + x(L), j + 2x(L) - y(L)}(L),
\end{equation}
where $x(L)$ and $y(L)$ are as defined  above.
The graded Euler characteristic of this invariant is equal to the normalized
Jones polynomial. More precisely,
\begin{small}
\begin{equation}\label{mainequation}
\sum_{i,j\in\mathbb{Z}}(-1)^{i}q^{j}\dim\mathcal{H}^{i,j}(L) = (q^{-1} + q)V(L)_{t^{1/2} = -q}.
\end{equation}
\end{small}

It is worth mentioning here  that the  graded Euler characteristic is indeed  the evaluation, at $t=-1$, of the  two-variable
Khovanov polynomial  invariant of the link  defined first in \cite{Ba} as follows:
\begin{equation}\label{polynomial}
\Kh(L)(t,q) = \sum_{i,j\in\mathbb{Z}}t^{i}q^{j}\dim\mathcal{H}^{i,j}(L).
\end{equation}

\subsection{Exact Sequence}

It is clear that $\mathcal{\overline{C}}(L_{0})$ and
$\mathcal{\overline{C}}(L_{1})[-1]\{-1\}$ are subcomplexes of
$\mathcal{\overline{C}}(L)$ and form  a short exact sequence
\[
0 \rightarrow \mathcal{\overline{C}}(L_{1})[-1]\{-1\} \rightarrow
\mathcal{\overline{C}}(L) \rightarrow \mathcal{\overline{C}}(L_{0}) \rightarrow 0,
\]
with degree preserving maps. This induces the following exact sequences:

\begin{small}
\begin{equation}\label{mainsequence}
\cdots\rightarrow \overline{\mathcal{H}}^{i-1,j}(L_{0}) \stackrel{\delta}{\rightarrow}
\overline{\mathcal{H}}^{i,j}(L_{1})[-1]\{-1\}\rightarrow \overline{\mathcal{H}}^{i,j}(L)
\rightarrow \overline{\mathcal{H}}^{i,j}(L_{0})\stackrel{\delta}{\rightarrow}
\overline{\mathcal{H}}^{i+1,j}(L_{1})[-1]\{-1\}\rightarrow \cdots
\end{equation}
\end{small}

Consequently, we obtain the following long exact sequences in  homology as a result of
the facts that $x(L_{0}) = x(L), y(L_{0}) = y(L) - 1, x(L_{1}) = x(L) + e$ and
 $y(L_{1}) = y(L)-e-1$:

\begin{equation}\label{positive}
\cdots\rightarrow \mathcal{H}^{i-1,j-1}(L_{0}) \stackrel{\delta}{\rightarrow}
\mathcal{H}^{i-e-1,j-3e-2}(L_{1})\rightarrow \mathcal{H}^{i,j}(L)\rightarrow
\mathcal{H}^{i,j-1}(L_{0})\stackrel{\delta}{\rightarrow}
\mathcal{H}^{i-e,j-3e-2}(L_{1})\rightarrow \cdots,
\end{equation}
where $e$ denotes the difference between the number of negative crossings in
$L_{1}$ and the number of such crossings in $L$.

We recall from \cite{Kh2} that a link $L$ is said to be  \begin{sl} $H$-thin \end{sl}   if its
Khovanov homology $\mathcal{H}(L)$ is supported on the two diagonals that satisfy
$j -2i = -\sigma(L) \pm 1$. It has been proved in \cite{L} that all alternating
links are $H$-thin. This fact has been later generalized to quasi-alternating
links in \cite[Theorem\,1]{MO}.

\begin{rem}\label{thin}
The result of \cite[Theorem\,1]{MO} shows   that the reduced Khovanov
homology of any quasi-alternating link is $\sigma$-thin over $\mathbb{Z}$. Indeed, this
 implies that the unreduced rational Khovanov homology  that we consider in this paper
is $H$-thin.
\end{rem}

\subsection{Lee Invariant}
The Lee invariant is a variant of the rational Khovanov homology
obtained from the same underlying complex with a  differential $\Phi$ of degree
difference $(1,4)$. We let $H^{i}(L)$ to denote the $i$-th homology
group of the complex $\mathcal{C}(D)$ using Lee's differential. We summarize all
results about this invariant needed in the sequel in the following
proposition:
\begin{prop}\cite[Section\,4]{L}\label{lee}
Let $L$ be an oriented link with $k$ components $L_{1},L_{2}, \ldots
,L_{k}$. Then
\begin{enumerate}
\item $\dim(H(L)) = 2^{k}$.
\item There is a generator of homology in degree $i$ such that

\[ \dim H^{i}(L) = 2 \times \left|\{ E\subset \{2,3,\ldots,n\} | \left(\sum_{l\in E, m\notin E}\lk(L_{l},L_{m}) =i\right)\}\right|
\] where the set $ \{1, 2, \ldots , n\}$ indexes the set of components
of the link $L$. The linking number $\lk(L_{l},L_{m})$ is the linking number (for the
original orientation) between the components $L_{l}$ and $L_{m}$.
\item For any $H$-thin link $L$, we have a spectral sequence converging to $H(L)$
with
$E_{1}^{s,t} = \mathcal{H}^{s+t,4s}(L)$ and $E_{2} = H(L)$. In other words, \[
H(L) \cong \frac{\Ker (\Phi:  \mathcal{H}(L) \rightarrow \mathcal{H}(L))}
{\Image (\Phi: \mathcal{H}(L) \rightarrow \mathcal{H}(L))}.
\]
\end{enumerate}
\end{prop}

Our main goal in this paper is to prove the following theorem and give some consequences in the last section.

\begin{thm}\label{main}
The length of any gap in the differential grading of the Khovanov homology of any quasi-alternating
link is one.
\end{thm}
\section{Proof of the Main Theorem}

The crucial step in the proof of Theorem \ref{main} relies on showing that the Lee homology of any quasi-alternating link
is bi-graded and therefore the Khovanov homology of such a link satisfies the Knight Move Conjecture  \cite[Conjecture\,1]{Ba}.

A simple  way to describe the Khovanov homology of a given  link is to locate a non negative
integer $n$ at the point $(i,j)$, where $n$ represents the dimension of
$\mathcal{H}^{i,j}(L)$ over $\mathbb{Q}$. An example is illustrated in Table \ref{Khtrefoil}.
Based on the this description of Khovanov homology, we introduce the following definition.
\begin{defn}
For a given link $L$,
\begin{enumerate}
\item We say that the differential grading of the Khovanov homology of $L$ has a gap of length
$s$ if there are $s$ consecutive columns of the Khovanov homology of the link $L$
with zero entries and these columns correspond to vertical lines of $x$-intercepts $i,i+1,i+2,\ldots,i+s-1$
such that $i_{\min} +1 \leq  i< i+1< i+2 <\ldots<i+s-1 \leq i_{\max} -1$. This is
equivalent to say that $\mathcal{H}^{i-1}(L)\neq 0$ and $\mathcal{H}^{i+s}(L)\neq 0$ while
$\mathcal{H}^{m}(L) = 0$ for $m=i,i+1,\ldots,i+s-1$.
\item We say that the quantum grading of the Khovanov homology of $L$ has a gap of length
$s$ if there are $s$ consecutive rows of the Khovanov homology of the link $L$
with zero entries and these rows correspond to horizontal lines of $y$-intercepts
$j,j+2,j+4,\ldots,j+2s-2$ such that
$j_{\min} +2 \leq  j< j+2< j+4 <\ldots<j+2s-2 \leq j_{\max} -2$.
\item We say that the the Jones polynomial of the link $L$ has a gap of length $s$
if this polynomial has two monomials $t^{i}$ and $t^{i+s+1}$ of nonzero coefficients
and all the monomials $t^{m}$ have zero coefficients for $m=i+1,i+2,\ldots,i+s$.
The gap in the Kauffman bracket polynomial of the link $L$ can be defined in
a similar manner.
\item We say that there is a gap between the two polynomials $f(A)$ and $g(A)$
of length $s$ if either $M = \maxdeg(f(A)) < m =\mindeg(g(A)) $ or
$M = \maxdeg(g(A)) < m = \mindeg(f(A))$ and all the monomials
$A^{M+1},A^{M+2},\ldots, A^{M+s}=A^{m-1}$ have zero coefficients in the
two polynomials.
\end{enumerate}
\end{defn}

To illustrate this definition, we consider  the Khovanov homology of the trefoil knot $3_1 $ given in Table \ref{Khtrefoil}. This homology  has a gap of length one in the differential grading as $\mathcal{H}^{-1}(3_1)=0$
while both $\mathcal{H}^{-2}(3_1)$ and $\mathcal{H}^{0}(3_1)$ are nontrivial. Also, the Jone polynomial of the trefoil that is known to be $-t^{-4} +t^{-3}+t^{-1}$ has a gap of length one as the monomial $t^{-2}$ has zero coefficient while the monomials $t^{-1}$ and $t^{-3}$ have nonzero coefficients.

\begin{small}
\begin{table}
\begin{tabular}{|c||c|c|c|c|}
\hline
\diagbox{$j$}{$i$} &-3&-2&-1&0\\
\hline
-1&&&&1\\
\hline
-3&&&&1\\
\hline
-5&&1&&\\
\hline
-7&&&&\\
\hline
-9&1&&&\\
\hline
\end{tabular}
\vspace{0.2cm}
\caption{\label{Khtrefoil} The Khovanov homology of the trefoil knot.}
\end{table}
\end{small}

The signatures of the links $L, L_{0}$ and $L_{1}$ are related according to the following lemma if the link $L$ is quasi-alternating at the crossing $c$.
\begin{lemm}\label{signature}\cite[Lemma\,3]{MO}
Let $L$ be a quasi-alternating link and $L_{0}$ and $L_{1}$ be its smoothings at the crossing $c$
where it is quasi-alternating. For the triple $(L, L_0,L_1)$, we have
\[
\sigma(L_{0}) = \sigma(L) + 1
\]
and
\[
\sigma(L_{1}) = \sigma(L) + e,
\]
where $e$ is as defined in  Section 2.
\end{lemm}

Unlike Khovanov homology which is bi-graded as explained in Section 2,  Lee homology theory is known to be filtered rather than bi-graded.  However, the Lee homology of a
quasi-alternating link can be endowed with a bi-grading structure as shown in the following theorem.

\begin{thm}\label{grading}
If $L$ is a quasi-alternating link, then the Lee homology $H(L)$ is bi-graded and
$H^{i,j}(L) \cong H^{i,j+2}(L)\cong \mathbb{Q}^{n}$ for some non-negative
integer $n$ and for any $i$ and $j$ such that $j-2i = -\sigma(L) -1$.
\end{thm}

\begin{proof}
The link $L$  is  $H$-thin as  it is quasi-alternating,  \cite[Theorem\,1]{MO}. Consequently, the bi-grading of $\mathcal{H}(L)$
induces a bi-grading on $H(L)$  as a result of Part 3 of
Proposition \ref{lee}. We use induction on the determinant of the link $L$ to
prove the second claim in the theorem. It is clear that the result holds when $\det(L) =1$ since
the only quasi-alternating link of determinant 1 is the unknot \cite[Prop.\,3.2]{G2}.
Now suppose that  the result holds for all links of determinant less than $m = \det(L)$.
In particular, it holds for the links $L_{0}$ and $L_{1}$. It is easy to see that
if $H^{i}(L) \neq 0$ then $H^{i\pm 1}(L) = 0$ as a result of Part 2 of Proposition
\ref{lee}. If we use this fact in the long exact sequence of the cohomology groups
$H(L), H(L_{0})$ and $H(L_{1})$ that is mentioned for the first time in the proof
of \cite[Theorem\,4.2]{L}, then we obtain

\begin{align}\label{leesequence}
0 \longrightarrow \overline{H}^{i-1}(L_{0})\longrightarrow \overline{H}^{i-1}(L_{1})
\longrightarrow \overline{H}^{i}(L)\longrightarrow \overline{H}^{i}(L_{0})
\longrightarrow \overline{H}^{i}(L_{1})\longrightarrow 0.
\end{align}

As a result of the facts that $x(L_{0}) = x(L), y(L_{0}) = y(L) - 1,
x(L_{1}) = x(L) + e$ and $y(L_{1}) = y(L)-e-1$, we obtain the following exact sequences:

\begin{align}\label{zero}
0 \longrightarrow H^{i-1}(L_{0})\longrightarrow H^{i-e-1}(L_{1})
\longrightarrow H^{i}(L)\longrightarrow H^{i}(L_{0})\longrightarrow H^{i-e}(L_{1})
\longrightarrow 0.
\end{align}
 and
\begin{small}
\begin{align}\label{first}
0 \longrightarrow H^{i-1,j-1}(L_{0})\longrightarrow H^{i-e-1,j-3e-2}(L_{1})
\longrightarrow H^{i,j}(L)\longrightarrow H^{i,j-1}(L_{0})
\longrightarrow H^{i-e,j-3e-2}(L_{1})\longrightarrow 0
\end{align}
\end{small}

The above exact sequence still holds if we replace $j$ by $j+2$ as follows:
\begin{small}
\begin{align}\label{second}
0 \longrightarrow H^{i-1,j+1}(L_{0})\longrightarrow H^{i-e-1,j-3e}(L_{1})
\longrightarrow H^{i,j+2}(L)\longrightarrow H^{i,j+1}(L_{0})
\longrightarrow H^{i-e,j-3e}(L_{1})\longrightarrow 0
\end{align}
\end{small}

It is easy to see that $H^{i-1,j+1}(L_{0}) = 0 = H^{i-e,j-3e-2}(L_{1})$ since
$(j+1)-2(i-1) = j -2i + 3 = -\sigma(L) -1 + 3 = -\sigma(L_{0}) + 3 \neq -\sigma(L_{0}) \pm 1$
and $(j-3e-2)-2(i-e) = j-2i -e -2 = -\sigma(L) -1 -e-2=-\sigma(L_{1})-3 \neq -\sigma(L_{1})
\pm 1$, respectively provided that $j-2i = -\sigma(L) -1 $. Also as a result of the induction
hypotheses on $L_{0}$ and $L_{1}$, we obtain $H^{i-e-1,j-3e-2}(L_{1}) \cong
H^{i-e-1,j-3e}(L_{1})$ and $H^{i,j-1}(L_{0}) \cong H^{i,j+1}(L_{0})$ since $(j-3e)-2(i-e-1)
= j-2i-e+2=-\sigma(L)-1-e+2=-\sigma(L_{1})+1, (j-3e-2)-2(i-e-1) = j-2i-e=-\sigma(L)-1-e=
-\sigma(L_{1})-1$ and $(j-1)-2i = j -2i -1 = -\sigma(L)-1-1=-\sigma(L_{0})-1, (j+1)-2i =
j -2i +1 = -\sigma(L)-1+1=-\sigma(L_{0})+1$. Now we have four cases to consider:

\begin{enumerate}
\item If $H^{i-e-1,j-3e-2}(L_{1}) \cong H^{i-e-1,j-3e}(L_{1}) = 0$ and
$H^{i,j-1}(L_{0}) \cong H^{i,j+1}(L_{0}) = 0$, then $H^{i,j+2}(L) = 0 = H^{i,j}(L)$.
\item If $H^{i-e-1,j-3e-2}(L_{1}) \cong H^{i-e-1,j-3e}(L_{1}) \neq 0$ and
$H^{i,j-1}(L_{0}) \cong H^{i,j+1}(L_{0}) \neq 0$, then as a result of the
fact that if $H^{i}(L) \neq 0$ then $H^{i\pm 1}(L) = 0$ obtained from Part 2
of Proposition \ref{lee} we conclude that $H^{i-1,j-1}(L_{0}) = 0 =
H^{i-e,j-3e}(L_{1})$. This implies that $ H^{i,j}(L)\cong H^{i,j-1}(L_{0})
\oplus  H^{i-e-1,j-3e-2}(L_{1}) \cong H^{i,j+1}(L_{0}) \oplus  H^{i-e-1,j-3e}(L_{1})
\cong  H^{i,j+2}(L)$.
\item If $H^{i-e-1,j-3e-2}(L_{1}) \cong H^{i-e-1,j-3e}(L_{1}) = 0$ and
$H^{i,j-1}(L_{0}) \cong H^{i,j+1}(L_{0}) \neq 0$,  then as a result of the
fact that if $H^{i}(L) \neq 0$ then $H^{i\pm 1}(L) = 0$ obtained from Part 2
of Proposition \ref{lee} we conclude that $H^{i-1,j-1}(L_{0}) = 0$. In particular,
we obtain $ H^{i,j-1}(L_{0})\cong H^{i,j}(L) \oplus  H^{i-e,j-3e-2}(L_{1})
\cong H^{i,j}(L)$ and  $H^{i,j+1}(L_{0}) \cong H^{i,j+2}(L) \oplus  H^{i-e,j-3e}(L_{1})$.
To prove the required result it is enough to prove that $H^{i-e,j-3e}(L_{1}) = 0$.
We prove this by contradiction. If we assume that $H^{i-e,j-3e}(L_{1}) \neq 0$, then
$H^{i-e}(L_{1}) \neq 0$ as a result of the fact that $H^{i-e}(L_{1}) \cong
H^{i-e, j-3e+2}(L_{1})\oplus H^{i-e, j-3e}(L_{1})$ obtained from Part 3 of
Proposition \ref{lee} because $L_{1}$ is $H$-thin as it is quasi-alternating.
Now according to the long exact sequence \ref{zero} and the assumption in this
case we conclude that $H^{i}(L_{0}) \cong H^{i}(L) \oplus  H^{i-e}(L_{1})
\cong \left(H^{i,j}(L)\oplus H^{i,j+2}(L)\right) \oplus \left(H^{i-e,j-3e}(L_{1})
\oplus H^{i-e,j-3e+2}(L_{1})\right)$. Thus we obtain $0 = H^{i-e,j-3e+2}(L_{1})
\cong H^{i-e,j-3e}(L_{1})$, where the isomorphism follows by the induction
hypothesis on the link $L_{1}$.
\item If $H^{i-e-1,j-3e-2}(L_{1}) \cong H^{i-e-1,j-3e}(L_{1}) \neq 0$ and
$H^{i,j-1}(L_{0}) \cong H^{i,j+1}(L_{0}) = 0$. This case can be  discussed using
a similar argument as in the previous case.
\end{enumerate}
\end{proof}
\begin{rem}\label{gradingnew}
In fact, if the $E_{2}$- and the $E_{\infty}$-pages of the spectral sequence
of the double complex $(\mathcal{C}(L), d(L), \Phi(L))$ are isomorphic, then
the Lee homology $H(L)$ is bi-graded since these two pages are isomorphic to
$\frac{\Ker (\Phi:  \mathcal{H}(L) \rightarrow \mathcal{H}(L))}
{\Image (\Phi: \mathcal{H}(L) \rightarrow \mathcal{H}(L))}$ and $H(L)$, respectively.
As a result of this fact, we can define the Lee polynomial invariant of  such a link as follows:
\begin{equation}\label{leepolynomial}
\Le(L)(t,q) = \sum_{i,j\in\mathbb{Z}}t^{i}q^{j}\dim H^{i,j}(L).
\end{equation}
\end{rem}

As a consequence of  Theorem \ref{grading} and Part 3 of Proposition \ref{lee},
we establish the Knight Move Conjecture \cite[Conjecture\,1]{Ba} for the class of
quasi-alternating links. This generalizes \cite[Theorem\,4.5]{L} to this class of links.
\begin{coro}
Let $L$  be  an oriented quasi-alternating link with  components $L_{1},L_{2},
\ldots,L_{n}$. Let  $l_{jk}$ denote  the linking number of  the components $L_{j}$ and $L_{k}$.
Then, we have 

\[ \Le(L)(t,q) = q^{-\sigma(L)}(q^{-1} +q)(\sum_{E\subset \{2,3,\ldots,n\}}
(tq^{2})^{\sum_{j\in E, k\notin E} 2l_{jk}}).\] Therefore, we obtain 
\[
\Kh(L)(t,q) = q^{-\sigma(L)}\{(q^{-1} +q)(\sum_{E\subset \{2,3,\ldots,n\}}
(tq^{2})^{\sum_{j\in E, k\notin E} 2l_{jk}})+ (q^{-1} +tq^{2}.q)\Kh^{'}(L)(tq^{2})
\},
\]
for some polynomial $Kh^{'}(L)$.
\end{coro}

One  can define two types of \textsl{breadth} for the Khovanov homology of
any link $L$ based on the differential and quantum gradings of Khovanov homology. These two
types are denoted hereafter  by $\breadth_{i}(\mathcal{H}(L))$ and
$\breadth_{j}(\mathcal{H}(L))$ and  are defined to be the difference
between the maximum and minimum corresponding gradings. For instance, the $\breadth_{i}(\mathcal{H}(3_1))=3$ while
$\breadth_{j}(\mathcal{H}(3_1))=8$ as one  can see from  Table \ref{Khtrefoil}.
In the case of quasi-alternating links, these two types
of breadth are related as follows.
\begin{coro}\label{breadth1}
If $L$ is a quasi-alternating link, then $\breadth_{j}(\mathcal{H}(L)) =
2\breadth_{i}(\mathcal{H}(L)) + 2$.
\end{coro}
\begin{proof}
In the above description of Khovanov homology, we claim that
the most left-bottom nonzero entry that comes right after any gap or at the beginning of the
Khovanov homology of $L$ will be on the lower diagonal with $j-2i = -\sigma(L) - 1$
and the most right-top nonzero entry right before any gap or at the end of the Khovanov
homology of $L$ will be on the upper diagonal with $j-2i = -\sigma(L) + 1$. Each one of
these two entries is the only entry in that row.

We show the first case of the claim and the second case of the claim follows using
a similar argument. Suppose that the most left-bottom entry shares the same row with
another entry to the right. In this case, the leftmost  entry will be on the upper
diagonal and therefore it will be the only entry
to survive in $H^{i}(L)$ in that column, where $i$ is the differential grading of the most left-bottom
entry. This is impossible according to the result of Theorem \ref{grading}.

Thus the minimum quantum grading $j_{\min}$ is related to the minimum differential grading
$i_{\min}$ by $j_{\min}-2i_{\min} = -\sigma(L) -1$ and the maximum quantum grading
$j_{\max}$ is related to the maximum differential grading $i_{\max}$ by
$j_{\max}-2i_{\max} = -\sigma(L) +1$. Now the result follows because
$\breadth_{j}(\mathcal{H}(L)) = j_{\max} - j_{\min} = (2i_{\max} -\sigma(L) + 1) -
(2i_{\min} - \sigma(L) - 1) = 2(i_{\max} - i_{\min}) + 2 =
2\breadth_{i}(\mathcal{H}(L)) + 2$.
\end{proof}

We can give an explicit formula for the Jones polynomial of any quasi-alternating
link in terms of the integers that are given in the above description of its
Khovanov homology. Indeed, these  integers can be expressed in terms of the dimensions
of the Khovanov and Lee homologies of the given link. More precisely, if  $a_{l}$ is  the  $(i_{\min}+l-1,j_{\min}+2l-2)$-entry of   the lower diagonal in the above
description of $\mathcal{H}(L)$ after taking the quotient by  $H(L)$. Then we have
\begin{lemm}\label{integers}
If $L$ is a quasi-alternating link, then
\begin{equation*}
a_{l} =
\begin{cases}
\dim \mathcal{H}^{i_{\min}}(L) - \dim H^{i_{\min}}(L), & \text{if} \ {l= 1},
      \\
\dim \mathcal{H}^{i_{\min}+l-1}(L) - \dim H^{i_{\min}+l-1}(L) - a_{l-1}, & \text{if} \ {1 < l < i_{\max}-i_{\min} }
\\
\dim \mathcal{H}^{i_{\max}}(L) - \dim H^{i_{\min}}(L), & \text{if} \ { i = i_{\max}-i_{\min}} .
     \\
\end{cases}
\end{equation*}
\end{lemm}
\begin{proof}
The result is a straightforward consequence of the fact that $L$ is  $H$-thin and Part 3 of Proposition \ref{lee}.
\end{proof}

\begin{coro}
If the link $L$ is quasi-alternating, then its Jones polynomial is given by
\begin{multline}\label{mainnew}
V_{L}(t)  = [(-1)^{i_{\min}}a_{1}q^{j_{\min}+1} + (-1)^{i_{\max}}a_{(i_{\max}-i_{\min})}
q^{j_{\max}-1} \\  + \sum_{j} n_{j}q^{j}+ \sum_{l=1}^{\frac{j_{\max}-j_{\min}-4}{2}}
(-1)^{l+i_{\min}}(a_{l}+a_{l+1})q^{j_{\min}+2l+1}]_{q = - \sqrt{t}},
\end{multline}
where the first sum is taken over all $j_{\min}+1 \leq j \leq j_{\max}-1$ such that
$H^{i,j+1}(L) \cong H^{i,j-1}(L) \cong \mathbb{Q}^{n_{j}}$ for some positive integer $n_{j}$.
\end{coro}

\begin{proof}
The entries at each column of
$H(L)$ contribute $n_{j}(q^{m}+q^{m+2}) = n_{j}q^{m+1}(q^{-1}+q)$ to $\Kh(-1,q)$ for some
integer $m$. Now the first sum follows after we divide by $q^{-1}+q$ and go over all the
columns of $H(L)$.

As a result of Part 3 of Proposition \ref{lee}, we conclude that there is a pairing (knight move)
in the entries of $\mathcal{H}(L)/H(L)$. In particular, the $(i_{\min}+l-1,j_{\min}+2l-2)$-entry of the lower diagonal is equal to the $(i_{\min}+l,j_{\min}+2l+2)$-entry for $1 \leq l \leq i_{\max}-i_{\min}$. It is not too hard to see
that the $(i_{\min}+l-1,j_{\min}+2l-2)$-entry is simply $a_{l}$ as a result of Lemma \ref{integers}. These two entries contribute
$(-1)^{i_{\min}+l-1} a_{l}(q^{j_{\min}+2l-2} -q^{j_{\min}+2l+2}) = (-1)^{i_{\min}+l-1}
 a_{l}(q^{-1}+q)(q^{j_{\min}+2l-1} - q^{j_{\min}+2l+1})$ to the polynomial $\Kh(-1,q)$.
The second sum follows after we divide by
$q^{-1}+q$ and go over all such entries.
\end{proof}

\begin{lemm}\label{coeff}
The coefficient $b_{j}$ of the monomial $q^{j}$ in the above formula is given by
\begin{equation*}
b_{j} =
\begin{cases}
(-1)^{i_{\min}}a_{1} + \frac{1}{2}\dim H^{i_{\min}}(L), & \text{if} \ {j= j _{\min}+1},
      \\
(-1)^{l+i_{\min}}(a_{i-i_{\min}}+a_{i-i_{\min}+1}) + \frac{1}{2}\dim H^{i_{\min}+ l}(L), & \text{if} \ {j_{\min} +1 < j < j_{\max}-1 }
\\
(-1)^{i_{\max}}a_{i_{\max}-i_{\min}} + \frac{1}{2}\dim H^{i_{\max}}(L), & \text{if} \ { j = j_{\max}-1},
     \\
\end{cases}
\end{equation*}
where $ l = \frac{j-j_{\min}-1}{2}$. Moreover, $b_{j}$ is zero only if the two terms that define $b_{j}$ are equal to zero.
\end{lemm}
\begin{proof}
We have three cases to consider according to the value of $j$. We discuss the case $j_{\min} +1 < j < j_{\max}-1$
and the other cases follow using a similar argument. If the monomial $q^{j}$ appears in the two sums in Equation \ref{mainnew}, then we obtain $n_{j} = \frac{1}{2}\dim H^{i}(L)$ with $j-2i = -\sigma(L)$
and $j = j_{\min}+2l+1$ for some $l$. Now if we use the fact that $j_{\min}- 2i_{\min} = -\sigma(L)-1$, we conclude that $i = i_{\min} + l$ and hence the result of the first part follows and the second part follows easily from the fact that $i$ has to be even if $\dim H^{i}(L)\neq 0$.

\end{proof}

\begin{prop}\label{breadth2case1}
Let $L$ be a quasi-alternating link. The differential grading of $\mathcal{H}(L)$ has a gap
of length $s$ if and only if $V_{L}(t)$ has a gap of length $s$.
\end{prop}
\begin{proof}

The result can be proved easily from the following three equivalences  which  are  direct consequences of  Lemma \ref{coeff}
\begin{enumerate}
\item $\dim \mathcal{H}^{i_{\min}}(L) = a_{1} + \dim H^{i_{\min}}(L) \neq 0 $ is equivalent to $a_{1} + \frac{1}{2}\dim H^{i_{\min}}(L) \neq 0$ which is also equivalent to $b_{j_{\min}+1} = (-1)^{i_{\min}}a_{1} + \frac{1}{2}\dim H^{i_{\min}}(L)\neq 0$.
\item $\dim \mathcal{H}^{i}(L) = a_{i-i_{\min}}+a_{i-i_{\min}+1} + \dim H^{i}(L) \neq 0$ is equivalent $ a_{i-i_{\min}}+a_{i-i_{\min}+1} + \frac{1}{2}\dim H^{i}(L) \neq 0$ which is also equivalent to $b_{j} =  (-1)^{l+i_{\min}}(a_{i-i_{\min}}+a_{i-i_{\min}+1}) + \frac{1}{2}\dim H^{i_{\min}+ l}(L) \neq 0$ where
$i_{\min} < i< i_{\max}$ and $j = 2i -\sigma(L)$.
\item $\dim \mathcal{H}^{i_{\max}}(L) = a_{i_{\max}-i_{\min}} + \dim H^{i_{\max}}(L) \neq 0$ is equivalent to $a_{i_{\max}-i_{\min}} + \frac{1}{2}\dim H^{i_{\max}}(L) \neq 0$ which is also equivalent to $b_{j_{\max}-1} = (-1)^{i_{\max}}a_{i_{\max}-i_{\min}} + \frac{1}{2}\dim H^{i_{\max}}(L) \neq 0$.
\end{enumerate}

\end{proof}

According to the long exact sequence in Equation \ref{positive}, the differential support of
$\mathcal{H}(L)$ is included in the union of the differential support of $\mathcal{H}(L_{0})$
and the differential support of $\mathcal{H}(L_{1})$ after doing the appropriate shift. We
denote the differential supports of $\mathcal{H}(L_{0})$ and $\mathcal{H}(L_{1})$ by $S_{0}$
and $S_{1}$, respectively. Also, we let $S\{k\}$ to denote the set $S$ shifted to the left by
$k$ that is $i \in S\{k\}$ if  $i+ k\in S$.

\begin{prop}\label{basic}
If there is a gap between the two polynomials $A\langle L_{0}\rangle$ and
$A^{-1}\langle L_{1} \rangle$ of length bigger than three, then
$|i_{1} - i_{0}|> 1$ for any $i_{0} \in S_{0}$ and $i_{1} \in S_{1}\{e+1\}$. In particular, the sets $S_{0}$ and $S_{1}\{e+1\}$  are disjoint and this implies that \[ \mathcal{H}^{i}(L) \cong \begin{cases}
{\mathcal{H}^{i}(L_{0})},& \mbox{if} \; i \in S_{0}, \vspace{1pt} \\
{\mathcal{H}^{i-e-1}(L_{1})}, & \mbox{if} \;  i \in S_{1}\{e+1\} \vspace{1pt},\\
{0}, & \mbox{otherwise}.
\end{cases}\]

\end{prop}

\begin{proof}
If there is a gap between the two polynomials $A\langle L_{0}\rangle$ and
$A^{-1}\langle L_{1} \rangle$ of length bigger than three, then either $M =
\maxdeg(A\langle L_{0}\rangle)< \mindeg(A^{-1}\langle L_{1}\rangle) = m$ or
$M = \maxdeg(A^{-1}\langle L_{1}\rangle) < \mindeg(A\langle L_{0}\rangle)= m$
such that the monomials $A^{M+1}, A^{M+1}, \ldots, A^{M+s}$ have zero coefficients
in both polynomials with $s$ at least seven as the powers of the monomials of
$\langle L\rangle$ differ by multiples of four. We discuss the first case and
the second case follows using a similar argument. Notice that we have

\begin{align*}
\left((-A)^{-3w(L)}(-A^{-2}-A^{2})\left\langle L \right\rangle\right)
& = (-t^{-1/2}-t^{1/2})V_{L}(t)_{t^{1/2}=-q}  = (q+q^{-1})V_{L}(t)_{t^{1/2}=-q}\\
& = \sum_{i,j\in\mathbb{Z}}(-1)^{i}q^{j}\dim\mathcal{H}^{i,j}(L).
\end{align*}

In addition,  since  $x(L_{0}) = x(L)$ and $ y(L_{0}) = y(L) -1$, we obtain $w(L_{0}) =
w(L) -1$ and

\begin{align*}
\left(A(-A)^{-3w(L)}(-A^{-2}-A^{2})\left\langle L_{0} \right\rangle\right)
& = \left(A(-A)^{-3(w(L_{0})+1)}(-A^{-2}-A^{2})\left\langle L_{0}
\right\rangle\right)_{A^{-2}=t^{1/2}=-q}\\ & = -t^{1/2}(-t^{-1/2}-t^{1/2})
V_{L_{0}}(t)_{t^{1/2}=-q}  \\ & = q(q+q^{-1})V_{L_{0}}(t)_{t^{1/2}=-q}\\
& = \sum_{i,j\in\mathbb{Z}}(-1)^{i}q^{j+1}
\dim\mathcal{H}^{i,j}(L_{0})\\
& = \sum_{i,j\in\mathbb{Z}}(-1)^{i}
q^{j}\dim\mathcal{H}^{i,j-1}(L_{0}).
\end{align*}

Moreover, from facts that $x(L_{1}) = x(L)+e$ and $ y(L_{1}) = y(L) -e -1$,
we obtain $w(L_{1}) = w(L) -2e-1$ where $e$ is as defined in Lemma \ref{signature} and

\begin{align*}
\left(A^{-1}(-A)^{-3w(L)}(-A^{-2}-A^{2})\left\langle L_{1} \right\rangle\right)
& = \left(A^{-1}(-A)^{-3(w(L_{1})+2e+1)}(-A^{-2}-A^{2})
\left\langle L_{1} \right\rangle\right)_{A^{-2}=t^{1/2}=-q}\\ & =
-t^{\frac{3e+2}{2}}(-t^{-1/2}-t^{1/2})V_{L_{1}}(t)_{t^{1/2}=-q}\\
&  = (-1)^{3e-1}q^{3e+2}(q+q^{-1})V_{L_{1}}(t)_{t^{1/2}=-q}\\ & =
\sum_{i,j\in\mathbb{Z}}(-1)^{i-e-1}q^{j+3e+2}
\dim\mathcal{H}^{i,j}(L_{1})\\
& = \sum_{i,j\in\mathbb{Z}}(-1)^{i}q^{j}
\dim\mathcal{H}^{i-e-1,j-3e-2}(L_{1}).
\end{align*}

As a direct consequence of the previous argument and the fact that $\langle L \rangle = A\langle L_{0}\rangle + A^{-1}\langle L_{1} \rangle$, we obtain the following   equality that can be also confirmed from the fact that the Euler characteristic of any long exact sequence vanishes
\begin{align*}
\sum_{i,j\in\mathbb{Z}}(-1)^{i}q^{j}\dim\mathcal{H}^{i,j}(L)
 = \sum_{i,j\in\mathbb{Z}}(-1)^{i}q^{j}
\dim\mathcal{H}^{i,j-1}(L_{0}) +  \sum_{i,j\in\mathbb{Z}}(-1)^{i}
q^{j}\dim\mathcal{H}^{i-e-1,j-3e-2}(L_{1}),
\end{align*}
with
\begin{equation}\label{onestep}
\maxdeg\left(\sum_{i,j\in\mathbb{Z}}(-1)^{i}q^{j}
\dim\mathcal{H}^{i,j-1}(L_{0})\right) < \mindeg\left(\sum_{i,j\in\mathbb{Z}}(-1)^{i}
q^{j}\dim\mathcal{H}^{i-e-1,j-3e-2}(L_{1})\right).
\end{equation}

Now, if $(i_{\max},j_{\max})$ are the coordinates
of the most right-top nonzero entry of $\mathcal{H}(L_{0})$ that is part of the
upper diagonal as pointed out in the proof
of Corollary \ref{breadth1}, then we obtain $\mathcal{H}^{i_{\max},j_{\max}+1}(L) \cong
\mathcal{H}^{i_{\max},j_{\max}}(L_{0})$ from the long exact sequence in Equation \ref{positive}.
The last result follows because $\dim \mathcal{H}^{i_{\max}-e-1,j_{\max}-3e-1}(L_{1}) = 0 = \dim \mathcal{H}^{i_{\max}-e,j_{\max}-3e-1}(L_{1})$ as they are the coefficients of $q^{j_{\max}+1}$ in $\sum_{i,j\in\mathbb{Z}}(-1)^{i}
q^{j}\dim\mathcal{H}^{i-e-1,j-3e-2}(L_{1})$ with

\begin{align*}
\mindeg\left(\sum_{i,j\in\mathbb{Z}}(-1)^{i}
q^{j}\dim\mathcal{H}^{i-e-1,j-3e-2}(L_{1})\right) & > \maxdeg\left(\sum_{i,j\in\mathbb{Z}}(-1)^{i}q^{j}
\dim\mathcal{H}^{i,j-1}(L_{0})\right)\\ & = j_{\max}+1.
\end{align*}

As a result of the fact $j_{\max} -2i_{\max}
= -\sigma(L_{0}) + 1$, we obtain $j_{\max} + 1 - 2i_{\max} = j_{\max} - 2i_{\max} + 1
= -\sigma(L_{0}) + 1 + 1 = -(\sigma(L_{0}) - 1) + 1 = -\sigma(L) +1$. In other words, the last
statement is equivalent to say that the above nonzero entry of the Khovanov homology of $L$ with coordinates
$(i_{\max},j_{\max}+1)$ will also lie on the upper diagonal of the Khovanov homology of $L$.

Similarly, if $(i_{\min},j_{\min})$ are the coordinates
of the most left-bottom nonzero entry of $\mathcal{H}(L_{1})$ that is part of the
lower diagonal as pointed in the proof
of Corollary \ref{breadth1}, then we obtain $\mathcal{H}^{i_{\min}+e+1,j_{\min}+3e+2}(L) \cong
\mathcal{H}^{i_{\min},j_{\min}}(L_{1})$ from the long exact sequence in Equation \ref{positive}.
The last result follows because $\dim \mathcal{H}^{i_{\min}+e,j_{\min}+3e+1}(L_{0}) = 0 = \dim \mathcal{H}^{i_{\min}+e+1,j_{\min}+3e+1}(L_{0})$ as they are the coefficients of $q^{j_{\min}+3e+2}$ in $\sum_{i,j\in\mathbb{Z}}(-1)^{i}
q^{j}\dim\mathcal{H}^{i,j-1}(L_{0})$ with

\begin{align*}
 j_{\min} +3e+2 & = \mindeg\left(\sum_{i,j\in\mathbb{Z}}(-1)^{i}
q^{j}\dim\mathcal{H}^{i-e-1,j-3e-2}(L_{1})\right)\\
& > \maxdeg\left(\sum_{i,j\in\mathbb{Z}}(-1)^{i}q^{j}
\dim\mathcal{H}^{i,j-1}(L_{0})\right).
\end{align*}

Since $j_{\min} -2i_{\min}
= -\sigma(L_{1}) + 1$, we obtain $j_{\min} + 3e+ 2 - 2(i_{\min}  + e+1) = j_{\min} - 2i_{\min} + e
= -\sigma(L_{1}) + e - 1 = -(\sigma(L_{1}) - e) - 1 = -\sigma(L) - 1$. In other words, the last
statement is equivalent to say that the above nonzero entry of the Khovanov homology of $L$ with coordinates
$(i_{\max}+e+1,j_{\max}+3e+2)$ will also lie on the lower diagonal of the Khovanov homology of $L$.

Thus, we conclude that $2(i_{\min}+e+1-i_{\max}) = j_{\min}+3e+2 + \sigma(L) + 1 - j_{\max} -\sigma(L) +1 = j_{\min} + 3e+ 2-j_{\max} + 2 \geq 2 + 2 = 4$, where the inequality follows from the fact that $j_{\min} + 3e +2-j_{\max} - 1> 0$ that is obtained from Equation \ref{onestep}. The last inequality implies that $i_{\min} + e+1 \geq i_{max} + 2$ and hence we conclude that $|i_{1}-i_{0}|\geq |(i_{\min} + e+1) - i_{max}| \geq 2$ 
for any $i_{0} \in S_{0}$ and $i_{1} \in S_{1}\{e+1\}$ as a result of  $i_{0} \leq i_{\max}$ and $ i_{\min}+e+1 \leq i_{1}$.

To compute the Khovanov homology of the link $L$, we need to consider only $i \in S_{0} \cup S_{1}\{e+1\}$ because otherwise we obtain $ 0 = \mathcal{H}^{i-e-1}(L_{1})\rightarrow \mathcal{H}^{i}(L)\rightarrow
\mathcal{H}^{i}(L_{0}) = 0$ from the long exact sequence and this implies that $\mathcal{H}^{i}(L) = 0$. The other two cases
are either $i \in S_{0}$ or $i \in S_{1}\{e+1\}$ but not both since the two sets $S_{0}$ and $S_{1}\{e+1\}$ are disjoint. We discuss the first case and the other case follows using a similar argument. If $i \in S_{0}$, then $i+e+1, i + e \notin S_{1}$ as a result of what we just proved.
Hence the result follows directly from the long exact sequence $ 0 = \mathcal{H}^{i-e-1}(L_{1})\rightarrow \mathcal{H}^{i}(L)\rightarrow
\mathcal{H}^{i}(L_{0}) \rightarrow \mathcal{H}^{i-e-1}(L_{1}) = 0$.

\end{proof}

\begin{coro}\label{leehomology}
If there is a gap between the two polynomials $A\langle L_{0}\rangle$ and
$A^{-1}\langle L_{1} \rangle$ of length bigger than three, then we have
\[  H^{i}(L) \cong \begin{cases}
{H^{i}(L_{0})},& \mbox{if} \; i \in S_{0}, \vspace{1pt} \\
{H^{i-e-1}(L_{1})}, & \mbox{if} \;  i \in S_{1}\{e+1\} \vspace{1pt},\\
{0}, & \mbox{otherwise}.
\end{cases}\]
\end{coro}
\begin{proof}
The $E_{\infty}$-page of the spectral sequence of the double complex $(\overline{\mathcal{C}}(L), d(L), \Phi(L))$  is simply the disjoint union
of the $E_{\infty}$-pages of the spectral sequences for the double complexes for the links $L_{0}$ and $L_{1}$. Thus the Lee homology of the link $L$ is simply the disjoint union of the Lee homologies of the links $L_{0}$ and $L_{1}$.
\end{proof}

\begin{coro}\label{resultofbasic1}
If $l_{0}\neq l-1$ or $l_{1}\neq l-1$ where $l,l_{0}$ and $l_{1}$ are the number of components of $L,L_{0}$ and  $L_{1}$  respectively and if there is a gap between the two polynomials
$A\langle L_{0}\rangle$ and $A^{-1}\langle L_{1} \rangle$, then this gap has to be of length
three.
\end{coro}
\begin{proof}
If the gap is of length bigger than three, then the results of Corollary \ref{leehomology} and Part 1 of Proposition \ref{lee} yields $2^{l} = \dim H(L) = \dim H(L_{0}) + \dim H(L_{1}) = 2^{l_{0}} + 2^{l_{1}}$.
Obviously, this  holds only if  $l_{0} = l-1 = l_{1}$.
\end{proof}

The following lemma is well-known about the Jones polynomial and it can be obtained by an argument similar to the proof of Proposition \ref{basic}.
\begin{lemm}\label{jonespolynomial}
The Jones polynomial of the link $L$ at the crossing $c$ satisfies one of the following skein relations:
\begin{enumerate}
\item If $c$ is a positive crossing of type I, then
$ V_{L}(t) = -t^{\frac{1}{2}}V_{L_{0}}(t)-t^{\frac{3e}{2}+1}V_{L_{1}}(t)$.
\item If $c$ is a negative crossing and of type I, then
$ V_{L}(t) = -t^{\frac{3e}{2}-1}V_{L_{0}}(t)-t^{\frac{-1}{2}}V_{L_{1}}(t)$.
\item If $c$ is a positive crossing and of type II, then
$ V_{L}(t) = -t^{\frac{3e}{2}+1}V_{L_{0}}(t)-t^{\frac{1}{2}}V_{L_{1}}(t)$.
\item If $c$ is a negative crossing of type II, then
$ V_{L}(t) = -t^{\frac{-1}{2}}V_{L_{0}}(t)-t^{\frac{3e}{2}-1}V_{L_{1}}(t)$.
\end{enumerate}
\end{lemm}

\begin{prop}\label{new}
Let $L$ be a link and $c$ be a crossing of this link consisting of two
arcs of two different components. Then the gap between the polynomials $A\langle L_{0}\rangle$ and $A^{-1}\langle L_{1}\rangle$ is of length at most seven.
\end{prop}
\begin{proof}
We assume that $c$ is a crossing between the two components $L_{1}$ and $L_{2}$ of the link $L$. We use induction on the length of the sequence of crossing switches required to turn the sublink of $L$ consisting of the components $L_{1}$ and $L_{2}$ into disjoint union of two links.

This sequence is a subset of the set of all crossings between the two components $L_{1}$ and $L_{2}$ of the link $L$. Each such crossing in this set consists of two arcs one coming from the component $L_{1}$ and the other one from $L_{2}$. Therefore, this set can be written as a disjoint union of two subsets one contains all crossings with the arc from $L_{1}$ is above the arc from $L_{2}$ and the other way around for the second one. A choice of the above sequence can be made to be either one of these two subsets among many other possible choices. Without loss of generality, we can assume that the crossing $c$ is the first element in this sequence and the link $K$ is obtained from $L$ by switching the crossing $c$. Therefore, the two links $L$ and $K$ are identical except at the crossing $c$.

From the induction hypothesis, the result holds for the link $K$ at the crossing $c$. In particular, the gap between the polynomials $A^{-1}\langle K_{0}\rangle$ and $A\langle K_{1}\rangle$ is of length at most seven. The result follows directly if this gap is of length less than or equal to three or if $\mindeg(A\langle K_{1}\rangle) >  \maxdeg(A^{-1}\langle K_{0}\rangle)$ noting that the links $K_{0}$ and $L_{0}$ are identical and the links $K_{1}$ and $L_{1}$ are also identical. Thus we can assume that this gap is of length exactly seven and $\mindeg(A^{-1}\langle K_{0}\rangle) >  \maxdeg(A\langle K_{1}\rangle)$.

The choice of the skein relation in Lemma \ref{jonespolynomial} that we can apply at the crossing $c$ to evaluate $V_{L}(t)$ or $V_{K}(t)$ depends on the type of the crossing $c$ and the orientation of the two components $L_{1}$ and $L_{2}$ 
either in $K$ or $L$. The fact that the Jones polynomials of the same link with two different orientations are related by some phase gives us the freedom to choose the orientations on the two components $L_{1}$ and $L_{2}$ without affecting the length of the gap. Without loss of generality and choosing the appropriate orientations on the components $L_{1}$ and $L_{2}$, we can assume that the crossing $c$ is positive in the link $L$ and negative in the link $K$.

As a result of the assumption $\mindeg(A^{-1}\langle K_{0}\rangle) >  \maxdeg(A\langle K_{1}\rangle)$, we can conclude that $\mindeg(V_{L_{0}}(t))> \maxdeg (t^{\frac{3e}{2}} V_{L_{1}}(t))$. 
In this case, the gap between the polynomials $-t^{\frac{-1}{2}}V_{L_{0}}(t)$ and $-t^{\frac{3e}{2}-1} V_{L_{1}}(t)$ in the link $K$ is of length $ (\mindeg(V_{L_{0}}(t))-\frac{1}{2})-(\maxdeg(V_{L_{1}}(t))-1)-1 = \mindeg(V_{L_{0}}(t))-\maxdeg(V_{L_{1}}(t))- \frac{1}{2}$  and the gap between the polynomials $-t^{\frac{1}{2}}V_{L_{0}}(t)$ and $-t^{\frac{3e}{2}+1} V_{L_{1}}(t)$ in the link $L$ is of length $(\mindeg(V_{L_{0}}(t)) +\frac{1}{2})-(\maxdeg(V_{L_{1}}(t))+1)-1 = \mindeg(V_{L_{0}}(t))-\maxdeg(V_{L_{1}}(t))- \frac{3}{2}$. Thus the result follows since the length of the gap in $L$ is smaller than the length of the gap in the link $K$.

\end{proof}

As a consequence of Corollary \ref{resultofbasic1} and Proposition \ref{new}, we obtain the following Corollary:
\begin{coro}
Let $L$ be a link. If there is a gap between the
two polynomials $A\langle L_{0}\rangle$ and $A^{-1}\langle L_{1} \rangle$, then
it has to be of length at most seven.
\end{coro}

\begin{proof}[\textbf{Proof of Theorem \ref{main}}]
We use induction on the determinant of the given quasi-alternating link to prove the
statement of the theorem. If $\det(L)=1$, then the result holds since the only
quasi-alternating link of determinant one is the unknot \cite[Prop.\,3.2]{G2}.
Now, suppose that the result holds for all quasi-alternating links of determinant
less than the determinant of the link $L$. In particular, any gap in the differential gradings of
$\mathcal{H}(L_{0})$ and $\mathcal{H}(L_{1})$ is of length one.

We have two cases to consider. The first case if there is a gap between the
polynomials $A\langle L_{0}\rangle$ and $A^{-1}\langle L_{1} \rangle$ of length
seven or less and this includes the case of having  no gap. The second case corresponds
to  a gap  of length bigger  than seven. In any case, there is no cancellation
between these two polynomials as a result of $L$ being quasi-alternating. Thus
any gap in the Jones polynomial of the link $L$  is induced either by a gap in one
of the two polynomials of $L_0$ and $L_1$ or by  the gap between these two polynomials.

In the first case,  it is easy to see that the Jones polynomial of $L$ has no gap
of length bigger than one as a result of the facts that the above polynomials have no
gap of length bigger than seven from the induction hypothesis on the links $L_{0}$
and $L_{1}$ and that the gap between them is of length seven or less.  Now the result
follows from Proposition  \ref{breadth2case1} since having a gap in the differential grading of
$\mathcal{H}(L)$ of length bigger than one induces a gap of length bigger than one in
the Jones polynomial of $L$ and this contradicts  what we know already about the
Jones polynomial of $L$.

In the second case, the gap between the two polynomials is of length bigger than seven.
In this case and according to Corollary \ref{resultofbasic1}, we can assume  that $L$ is a link
of more than one component that is quasi-alternating at a crossing consisting of two
arcs of two different components with $l_{0} = l - 1 = l_{1}$. This is impossible according to Proposition
\ref{new} and hence the result follows.

\end{proof}
\section{Consequences of the Main Theorem}

In this section, we discuss some consequences and applications of Theorem \ref{main}. We start by stating  two  corollaries.
The first  of which establishes a weaker version of Conjecture 2.3 in \cite{CQ} and
the second one establishes a weaker version of Conjecture 3.8 in \cite{QC}.

\begin{coro}\label{basicnew}
If $L$ is a quasi-alternating link, then the length of any gap in the Jones
polynomial $V_{L}(t)$ is one.
\end{coro}
\begin{proof}
The result follows easily by  combining  Theorem \ref{main} with   Proposition \ref{breadth2case1}.
\end{proof}

\begin{coro}
If $L$ is a quasi-alternating link, then $\left\lceil {\frac{\breadth V_{L}(t)}{2}}
\right\rceil + 1\leq \det(L)$.
\end{coro}
\begin{proof}
The result follows from the fact that any gap in  the Jones polynomial of the link has
length one. This implies that the Jones polynomial has at least
$\left(\left\lceil {\frac{\breadth V_{L}(t)}{2}} \right\rceil+ 1\right)$ monomials  each of which  contributes an increment of at least one to the  value of the determinant
due to the fact that the Jones polynomial of a quasi-alternating  link is alternating.
\end{proof}

\begin{coro}
Let $L$ be  a quasi-alternating link. Then,
the length of any gap in the quantum grading of $\mathcal{H}(L)$ is one.
\end{coro}
\begin{proof}
Suppose there are two consecutive rows that correspond to horizontal lines of $y$-intercepts $j$ and $j+2$ of $\mathcal{H}(L)$
with zero entries such that $j_{\min} + 2 \leq  j < j+2 \leq j_{\max}-2$. In this case, this implies that
$\mathcal{H}^{i}(L) = 0$ for some $i$ with $j-2i = -\sigma(L) -1$. From the assumption,
we have $\mathcal{H}^{i,j}(L) \cong \mathcal{H}^{i-1,j}(L) \cong \mathcal{H}^{i,j+2}(L) \cong
\mathcal{H}^{i+1,j+2}(L) = 0$. As a result of the knight move, we obtain
$\mathcal{H}^{i-1,j-2}(L) \cong \mathcal{H}^{i,j+2}(L) =0 $ and $ \mathcal{H}^{i+1,j+4}(L)
\cong \mathcal{H}^{i,j}(L)   = 0$. This implies that
$\mathcal{H}^{i-1}(L) = \mathcal{H}^{i}(L) = \mathcal{H}^{i+1}(L) = 0$. Also,
as a result of the knight move we obtain $j_{\min} + 4\leq j  \leq j_{\max} - 4$.
Therefore, we conclude that $i_{\min}+2\leq i \leq i_{\max} -2$.  This contradicts
the fact that the length of any gap in the differential grading of $\mathcal{H}(L)$ is one.
\end{proof}

As another  application, we obtain a necessary condition for a Kanenobu knot  to be  quasi-alternating.

\begin{coro}
If $|p+q|> 6 $, then the Kanenobu knot $K(p,q)$ is not quasi-alternating.
\end{coro}
\begin{proof}

Any Kanenobu knots $K(p,q)$ with $|p+q|> 6 $ has gap in the differential grading of $\mathcal{H}(K(p,q))$
of length bigger than one as it was proved in \cite[Theorem\,1.3]{QM}.
\end{proof}
\begin{rem}
The above result can be also obtained in terms of $V_{L}(t)$. In particular,
it is proven in \cite{K} that $V_{K(p, q)}(t)  = (-t)^{p + q} (V_{K(0, 0)}(t) - 1)
+ 1 = (-t)^{p + q} ((t^{-2} - t^{-1} + 1 - t + t^2)^2 - 1) + 1$. It is easy to see that
if $|p+q|> 6$, then the above polynomial has a gap of length bigger than one. Thus
according to Corollary \ref{basicnew}, the knot $K(p,q)$ is not quasi-alternating. It is
known though that all these knots but finitely many of them are not quasi-alternating as
a result of Corollary 3.3 in \cite{QC}.
\end{rem}

Now we investigate the implications of the main result to the class of alternating
links being considered as a  special class of quasi-alternating links. We know that any prime  alternating
link that is not a $(2,n)$-torus link has no gap in its Jones polynomial as a result of
\cite[Theorem\,1(iv)]{Th}. In addition, it is known that all alternating links are $H$-thin  \cite[Theorem\,1.2]{L}. If we combine all these results, we obtain
\begin{coro}
If $L$ is a prime  alternating link that is not a $(2,n)$-torus link, then it has no gap in
the differential grading of $\mathcal{H}(L)$.
\end{coro}
\begin{proof}
If we assume that the differential grading of $\mathcal{H}(L)$ has a gap, then
Proposition  \ref{breadth2case1} implies   that the Jones polynomial
has a gap. This  contradicts \cite[Theorem\,1(iv)]{Th}.
\end{proof}
\begin{rem}
It is easy to see that if $L$ is a $(2,n)$-torus link, then it has only one gap in the
differential grading of $\mathcal{H}(L)$ according to the computations in \cite{Kh1}.
\end{rem}

\begin{coro}\label{breadth}
For any link $L$, we have $\breadth(V_{L}(t)) \leq \breadth(V_{L_{0}}(t)) + \breadth(V_{L_{1}}(t)) + 2$. In the case if $L$ is connected and alternating, then we obtain $c(L) \leq c(L_{0}) + c(L_{1})+ 2$, where $L_{0}$ and $L_{1}$ are the links obtained by smoothing any crossing of any reduced connected alternating diagram of the link $L$.
\end{coro}
\begin{proof}
If we combine the results of Corollary \ref{resultofbasic1} and Proposition \ref{new}, we obtain $ \breadth(\langle L \rangle ) \leq \breadth(\langle L_{0}\rangle) + \breadth(\langle L_{1}\rangle) + 8$.
Now the first result follows as a consequence of the fact that $\breadth (V_{L}(t)) = \frac{\breadth(\langle L\rangle)}{4}$. For the second result, we assume that $L$ is a reduced connected alternating diagram of the link $L$. In this case, we obtain $c(L) = \breadth(V_{L}(t)) \leq \breadth(V_{L_{0}}(t)) + \breadth(V_{L_{1}}(t)) + 2 = c(L_{0}) + c(L_{1}) +2 $ where the equalities follow as a consequence of \cite[Theorem\,2.10]{Ka} that implies $ \breadth(V_{L}(t)) = c(L)$ for any connected reduced alternating link diagram and the fact that the diagrams $L_{0}$ and $L_{1}$ are connected and alternating if $L$ is a reduced connected alternating diagram of given link. 
\end{proof}

Many questions arise naturally about the inequality in Corollary \ref{breadth}. We would like just to post the following two questions:
\begin{ques}
For what class of links does the inequality in Corollary \ref{breadth} hold?
\end{ques}
\begin{ques}
Is there a link $L$ for which the inequality in Corollary \ref{breadth} does not hold?
\end{ques}

An easy technique to produce new examples of quasi-alternating links from old ones was
introduced by Champanerkar and Kofman \cite[Page\,2452]{CK}. It basically replaces the
crossing $c$ where the link $L$ is quasi-alternating by an alternating rational tangle
of the same type. This technique has been later generalized not only to a single rational
tangle but also to a product of rational tangles  \cite[Definition\,2.5]{QCQ} and to non-algebraic alternating tangles \cite{CKa}.  The following proposition shows  how this twisting technique
affects the gaps in the differential grading.
\begin{prop}\label{twisting}
Let $L$ be a quasi-alternating link at some crossing $c$ and let $L^{*}$ be the link
obtained by replacing the crossing $c$ in $L$ by a product of rational tangles that
extends $c$. Then the number of gaps in the differential grading of $\mathcal{H}(L^{*})$ is less
than or equal the number of gaps in $\mathcal{H}(L)$. In particular, if there are no
gaps in differential grading of $\mathcal{H}(L)$, then it is also  the case for the differential grading
of $\mathcal{H}(L^{*})$.
\end{prop}
\begin{proof}
For a positive integer $n$, we let $L^{n}$ denote the link with the crossing $c$
replaced by an alternating integer tangle of $n$ vertical or horizontal crossings
of the same type. To prove our result, we first show that the result holds for $L^{n}$.
By writing the Kauffman bracket skein relations, one can easily prove that:
\begin{align*}
\langle L^{n} \rangle & = A^{n-1}\left(A \langle L_{0} \rangle + A^{-1}\langle L_{1}
\rangle \right) + \left(\sum_{i=1}^{n-1}(-1)^{i}A^{n-4i-2}\right)\langle L_{1}
\rangle, {\mbox{ and} }\\ \langle L^{n} \rangle & = A^{-n+1}\left(A \langle L_{0} \rangle
+ A^{-1}\langle L_{1} \rangle \right) + \left(\sum_{i=1}^{n-1}(-1)^{i}A^{-n+4i+2}\right)
\langle L_{1} \rangle,
\end{align*}
for the case of vertical and horizontal tangles, respectively. Now, we discuss the first
case. The second case can be treated in a similar manner. It is not too hard to see
that any gap in $\langle L^{n}\rangle $ is induced basically by a gap in $\langle L
\rangle$. Finally the result follows since any product of rational tangles can be
obtained by a sequence of integer tangles and the result of Proposition  \ref{breadth2case1}.

\end{proof}

Quasi-alternating links of braid index 3 and quasi-alternating Montesinos links have been
classified in \cite{B, I}, respectively. Now we apply the above result to these
two classes of quasi-alternating links to obtain:

\begin{coro}
All quasi-alternating Montesinos links and quasi-alternating links of braid
index 3 have no gap in the differential grading of their Khovanov homologies.
\end{coro}
\begin{proof}
The result follows as a result of Corollary \ref{twisting} and the fact that
the Jones polynomial of all quasi-alternating Montesinos links and quasi-alternating
links of braid index 3 have no gap (Theorem 3.10 and Theorem 4.3 in \cite{CQ}, respectively).
\end{proof}

Finally, we enclose our discussion with the following conjecture that is supported by the above results and
implies both Conjecture 2.3 in \cite{CQ} and Conjecture 3.8 in \cite{QC}. In particular, it implies that the
Jones polynomial of any prime quasi-alternating link that is not a $(2,n)$-torus link has
no gap and the breadth of the Jones polynomial of such a link is a lower bound of the
determinant of this link.
\begin{conj}\label{conjmain}
If $L$ is a prime quasi-alternating link that is not a $(2,n)$-torus link, then the differential grading
of $\mathcal{H}(L)$ has no gap.
\end{conj}

\end{document}